\documentclass[12pt]{amsart}
\usepackage{amsthm}
\usepackage{amssymb}
\usepackage{amsmath}
\numberwithin{equation}{section}
\newcommand{\bea}{\begin{eqnarray}}
\newcommand{\eea}{\end{eqnarray}}
\newcommand{\be}{\begin{eqnarray*}}
\newcommand{\ee}{\end{eqnarray*}}
\newtheorem{theorem}{Theorem}[section]
\newtheorem{lemma}{Lemma}[section]
\newtheorem{corollary}{Corollary}[section]

\newtheorem{example}{Example}[section]
\begin{document}
\title[Automorphism Groups of Idempotent Evolution Algebras]{On Automorphism Groups of \\ Idempotent Evolution Algebras}
\author{Songpon Sriwongsa}
\address{Department of Mathematics, Faculty of Science, King Mongkut's University of Technology Thonburi (KMUTT), Bangkok 10140, Thailand} 
\email{songpon.sri@kmutt.ac.th}
\author{Yi Ming Zou}
\address{Department of Mathematical Sciences, University of Wisconsin, Milwaukee, WI 53201, USA}
\email{ymzou@uwm.edu}
\maketitle
\footnote{Note added in this version: This paper was first submitted to Proc. of AMS in May 2019, since the editor was unable to find a referee for the paper for over two years, it was then withdrawn from PAMS upon the suggestion of the editor and submitted to LAA. After it was accepted by LAA and subsequently posted on arXiv, the authors were informed the existence of the reference \cite{Cos} by A. Viruel. Partial results of this paper were presented in the 42nd Australasian Conference on Combinatorial Mathematics and Combinatorial Computing (42ACCMCC) in December 2019.}

\begin{abstract}
We study the automorphism group of an idempotent evolution algebra, show that any finite group can be the automorphism group of an evolution algebra, and describe certain evolution algebras with given automorphism groups. In particular, we classify $n$-dimensional idempotent evolution algebras whose automorphism group is isomorphic to the symmetric group $S_n$, and classify idempotent evolution algebras with maximal diagonal automorphism subgroups.
\end{abstract}
\par
\medskip
\section{Introduction}
Evolution algebras are non-associative and commutative algebras motivated by the evolution laws of genetics  \cite{Tian1}. These algebras are related to different fields and present many interesting properties \cite{Cas1, Cas2, Cas3, Cas4, Casa1, Eld1, Tian1}. Here we consider finite dimensional evolution algebras over a field $\mathbb F$. According to \cite{Tian1}, an $n$-dimensional evolution algebra $\mathcal E$ over $\mathbb F$ can be defined by using a natural basis $e_1,...,e_n$ and a structure matrix $A =(a_{ij}),\;a_{ij}\in \mathbb{F}, \;1\le i, j\le n$, such that
\bea\label{e11}
\quad e_ie_j = 0,\;\mbox{if $i\ne j$}\quad \mbox{and}\quad(e_1^2, ..., e_n^2) = (e_1, ..., e_n)A.
\eea
Note that in general, natural bases are not unique and different natural bases lead to different structure matrices \cite{Eld1, Tian1}
\par\medskip
We call an evolution algebra $\mathcal E$ {\it idempotent} if $\mathcal{E}^2 = \mathcal E$. From (\ref{e11}), it is clear that
\bea\label{e12}
\qquad\quad\mathcal{E}^2 = \mathcal E\Leftrightarrow e_1^2, ..., e_n^2\; \mbox{form a basis of}\; \mathcal E\Leftrightarrow A\; \mbox{is nonsingular}. 
\eea
\par
We denote by $\mathcal{E}(A)$ the evolution algebra with the structure matrix $A$ if we need to specify $A$, and denote by $\Gamma_A$ (or $\Gamma_{\mathcal E}$) the graph whose adjacency matrix is obtained from $A$ by replacing all nonzero entries of $A$ by $1$. The vertices of $\Gamma_A$ will be just $e_1, ..., e_n$. We call $\Gamma_A$ the graph associated with $\mathcal E$. Note that $\Gamma_A$ is a digraph.
\par
The automorphism group of an idempotent finite dimensional evolution algebra has been studied in \cite{Eld1, Eld2} via the associated graph $\Gamma_A$. In particular, it was shown in \cite{Eld1} that the automorphism group of a finite dimensional idempotent evolution algebra is finite. 
\par
In section 2, we revisit the main results of \cite{Eld1, Eld2} on the automorphism groups of finite dimensional idempotent evolution algebras using a different approach to gain more insight on these automorphism groups. The classifications of evolution algebras in dimensions $\le 4$ have been attempted, however, the classification lists are long even for these low dimensions \cite{Cas2, Cas3} (even with incomplete classifications). In this paper, we consider the problem from a different viewpoint: classify evolution algebras with a given automorphism group.  Since the automorphism group of an idempotent evolution algebra is finite (usually not finite otherwise \cite{Eld2}), it is natural to consider idempotent evolution algebras associated with a given finite group. In section 3, we show that any finite group can be the automorphism group of an evolution algebra, and identify certain evolution algebras with given groups. In particular, we give a classification of $n$-dimensional idempotent evolution algebras that have $S_n$ as the automorphism group. In section 4, we show that the maximal diagonal automorphism subgroup that an $n$-dimensional idempotent evolution algebra can have is the cyclic group $C_{2^n - 1}$ of order $2^n - 1$. Under the assumption that the field $\mathbb F$ is algebraically closed of characteristic $0$, we show that, up to isomorphism, there exists only one $n$-dimensional idempotent evolution algebra whose diagonal automorphism subgroup is $C_{2^n - 1}$, and its automorphism group is $C_{2^n-1}\rtimes C_n$. 
\section{Automorphisms of an idempotent evolution algebra}
\par\medskip 
Let $\mathcal{E}$ (resp. $\mathcal{E}'$) be an idempotent evolution algebra with a natural basis $e_1, ..., e_n$ (resp. $e_1', ..., e_n'$) and the structure matrix $A =(a_{ij})$ (resp. $B = (b_{ij})$). If $\phi : \mathcal{E} \rightarrow \mathcal{E}'$ is an isomorphism, then $(\phi(e_1), ..., \phi(e_n))$ is a natural basis of $\mathcal{E}'$ with the structure matrix $A$. Let $P =(p_{ij})$ be the matrix of bases change in $\mathcal{E}'$ defined by $(\phi(e_1), ..., \phi(e_n)) = (e_1', ..., e_n')P$. Then we have the following from \cite{Tian1}:
\bea\label{e21}
BP^{(2)} = PA \quad\mbox{and}\quad B(P\ast P) = 0,
\eea
where $P^{(2)} = (p_{ij}^2)$, and $P\ast P = (c_{ij}^k)$ is an $n\times \frac{n(n-1)}{2}$ matrix whose rows are indexed by $k$ and columns are indexed by the pairs $(i,j)$ such that $1\le i<j\le n$. The entries are defined by $c_{ij}^k = p_{ki}p_{kj}, i<j$. Since we assumed that $\mathcal{E}'$ is idempotent, $B$ is nonsingular, so $P\ast P = 0$, which implies that for each row $k$ ($1\le k\le n$) of $P$, there exists exactly one nonzero element, say $p_{kk'}$. By the fact that $\det(P)\ne 0$, we see that there exists a permutation $\sigma\in S_n$ such that $\sigma(k')= k, 1\le k\le n$. Thus we have the following:
\begin{theorem}
Two idempotent evolution algebras $\mathcal{E}$ and $\mathcal{E}'$ are isomorphic if and only if there exists a permutation $\sigma\in S_n$ and an $n\times n$ matrix $P=(p_{ij})$, such that $p_{ij}\ne 0 \Leftrightarrow i=\sigma(j)$ and $BP^{(2)} = PA$.
\end{theorem}
\par
Now consider $\mathcal{G} = \mbox{Aut}(\mathcal{E})$ for an idempotent evolution algebra $\mathcal{E}$. For $g\in \mathcal{G}$, let $G = (g_{ij})$ be the matrix of $g$ with respect to a natural basis $e_1, ..., e_n$, that is, $g(e_i) = \sum_{i=1}^ng_{ki}e_k, 1\le i\le n$. Applying Theorem 2.1 to the setting $A=B, G = P$, we have the following (\cite {Eld1}, Corollary 4.7): 
\begin{corollary}   For any $g\in \mathcal{G}$, there exists an element $\sigma$ of the symmetric group $S_n$ such that  $g(e_i) = d_{i}e_{\sigma(i)}$ for some $0\ne d_{i} \in \mathbb{F}$, $1\le i\le n$. 
\end{corollary}
For $g\in \mathcal{G}$, let $D_g = \mbox{diag}(d_{1}, ..., d_{n})$ be the diagonal matrix determined by $g$, and let $P_{\sigma}$ be the permutation matrix corresponds to $\sigma$, where the $d_i$'s and the $\sigma$ are determined by $g$ as in Corollary 2.1. Then in matrix form, we have
\bea\label{e22}
g\;:\; (e_1, ..., e_n) \longrightarrow (e_1, ..., e_n)P_{\sigma}D_g.
\eea 
\par
Define $\phi :\mathcal{G} \longrightarrow S_n$ by $\phi(g) = \sigma$, where $\sigma$ corresponds to the permutation matrix $P_{\sigma}$ determined by $g$ as in (\ref{e22}). Note that
\bea\label{e23}
P_{\sigma}^{-1}\mbox{diag}(d_1, ..., d_n)P_{\sigma} = \mbox{diag}(d_{\sigma(1)}, ..., d_{\sigma(n)}).
\eea
This can be seen by reducing it to the case of a transposition. Note also that
\be
P_{\sigma_2}^{-1}P_{\sigma_1}^{-1}\mbox{diag}(d_1, ..., d_n)P_{\sigma_1}P_{\sigma_2}
 = \mbox{diag}(d_{\sigma_1\sigma_2(1)}, ..., d_{\sigma_1\sigma_2(n)}).
\ee
This is because if $P_{\sigma_1}^{-1}\mbox{diag}(d_1, ..., d_n)P_{\sigma_1} = \mbox{diag}(d'_{1}, ..., d'_{n})$, where $d'_i = d_{\sigma_1(i)}$, then 
\be
P_{\sigma_2}^{-1}\mbox{diag}(d'_1, ..., d'_n)P_{\sigma_2} &=& \mbox{diag}(d'_{\sigma_2(1)}, ..., d'_{\sigma_2(n)})\\
{} &=& \mbox{diag}(d_{\sigma_1\sigma_2(1)}, ..., d_{\sigma_1\sigma_2(n)}).
\ee
Thus for $g_1,g_2\in \mathcal{G}$, we have
\bea\label{e24}
(P_{\sigma_1}D_{g_1})(P_{\sigma_2}D_{g_2}) = P_{\sigma_1}P_{\sigma_2}D'= P_{\sigma_1\sigma_2}D'
\eea
for some diagonal matrix $D'$, which implies that $\phi$ is a group homomorphism and $\ker(\phi)$ consists of the diagonal automorphisms of $\mathcal E$. 
\par
Furthermore, if $g(e_i) = d_{i}e_{\sigma(i)}, 1\le i\le n$, then from $g(e_i^2) = d_i^2e_{\sigma(i)}^2$ and $e_i^2 = \sum_{j=1}^na_{ji}e_j$ (see (\ref{e11})), we have
\bea\label{e25}
\sum_{j=1}^na_{ji}d_je_{\sigma(j)} = d_i^2\sum_{j=1}^na_{j\sigma(i)}e_j =  d_i^2\sum_{j=1}^na_{\sigma(j)\sigma(i)}e_{\sigma(j)}.
\eea
Since $(e_1, ..., e_n)$ is a basis, (\ref{e25}) holds if and only if
\bea\label{e26}
 d_ja_{ji} = d_i^2a_{\sigma(j)\sigma(i)}, \;\forall \;i, j. 
\eea
Then since $d_i\ne 0,\;1\le i\le n$, $a_{ji}\ne 0$ if and only if $a_{\sigma(j)\sigma(i)}\ne 0$. Thus $g$ induces a graph automorphism of $\Gamma_A$ via $\sigma$. 
\par
Let Aut$(\Gamma_A)$ be the graph automorphism group of $\Gamma_A$. Then Aut$(\Gamma_A)$ is a subgroup of $S_n$.
\par
For our convenience, we recall the following from \cite{Tian1}. Let $\mathcal{E}$ be an arbitrary $n$-dimensional evolution algebra with a natural basis $e_1, ..., e_n$ and the structure matrix $A$. For a linear endomorphism $g$ of $\mathcal E$, let $G$ be the matrix of $g$ with respect to $e_1, ..., e_n$. Then
\bea\label{e27}
\qquad\mbox{Aut}(\mathcal{E}) =\{G\;|\; \det(G)\ne 0,AG^{(2)}=GA,\mbox{and}\; A(G\ast G) = 0\}.
\eea
\par
Note that for a diagonal matrix $D$, $A(D\ast D) = 0$ is always true since $D\ast D = 0$, and $D^{(2)} = D^2$. If $D = \mbox{diag}(d_1, ..., d_n)\in\ker(\phi)\subset\mbox{Aut}(\mathcal{E})$, then $AD^{(2)} = DA$ is equivalent to 
\bea\label{e28}
d_j^2a_{ij} = d_ia_{ij},\quad 1\le i,j\le n.
\eea
If $\mathcal{E}$ is idempotent, then $\det(A)\ne 0$, which implies that there exists a permutation $\tau\in S_n$ such that for each $1\le j\le n$, $a_{\tau(j)j}\ne 0$. This in turn implies that $d_j^2=d_{\tau(j)}$ by (\ref{e28}). Let the order of $\tau$ be $t$. If $t=1$, i.e. $\tau$ is the identity, then $d_j^2=d_{j}$, so all $d_j = 1$ since $d_j\ne 0$. If $t> 1$, then
\be
d_j^{2^t} = (d_j^{2})^{2^{t-1}} = d_{\tau(j)}^{2^{t-1}} = \cdots = d_{\tau^t(j)} = d_j, \;\forall\; 1\le j\le n. 
\ee
This implies that all $d_j$ are roots of  $x^{2^t - 1} -1$ since $d_j\ne 0$. 
\par
Let $t_A$ be the smallest such $t$. That is, $t_A$ is the minimum order of the permutations $\tau\in S_n$ such that $a_{\tau(1)1}\cdots a_{\tau(n)n}\ne 0$. Then our discussions have proved the following theorem, which summarizes the main results of \cite{Eld1, Eld2} on the automorphism group of a finite dimensional idempotent evolution algebra (cf. Theorem 4.8 of \cite{Eld1}, and Theorem 3.2 of \cite{Eld2}).
\begin{theorem} Let $\mathcal E$ be an $n$-dimensional idempotent evolution algebra with natural basis $e_1, ..., e_n$ and structure matrix $A$, let $\mathcal G$ be the automorphism group of $\mathcal E$, and let ${\mathcal D}\subset \mathcal G$ be the subgroup of diagonal automorphisms.
\par
(1) The subgroup $\mathcal D$ is a normal subgroup of $\mathcal G$. The diagonal entries of an element of $\mathcal D$ are roots of $x^{2^{t_A} - 1} - 1$. In particular, $\mathcal D$ is a finite group of odd order.
\par  
(2) The quotient group $\mathcal{G}/\mathcal{D}$ is isomorphic to a subgroup of Aut$(\Gamma_A)$. In particular, $\mathcal G$ is finite.
\par\medskip 
\end{theorem}
\par
It is clear that, in general, not every graph automorphism of the associated graph $\Gamma_A$ induces an automorphism of the evolution algebra $\mathcal E$. However, we have the following:
\begin{theorem}
Let $A = (a_{ij})$ be the adjacency matrix of a graph $\Gamma$ with $n$ vertices. If an evolution algebra $\mathcal E$ has $A$ as the structure matrix, then every element of Aut$(\Gamma)$ induces an element of $\mathcal{G} =\mbox{Aut}(\mathcal E)$. If in addition $A$ is nonsingular, that is, $\mathcal E$ is idempotent, then $\mathcal{G}/\mathcal{D} \cong\mbox{Aut}(\Gamma)$.
\end{theorem}
\begin{proof}
Let $\mathcal E$ be defined by $A$ with the natural basis $e_1, ..., e_n$. If $P_{\sigma} = (p_{ij})$ is the matrix of a permutation $\sigma\in S_n$ with respect to the basis $e_1, ..., e_n$, that is 
\be
\sigma \;:\; (e_1, ..., e_n) \longrightarrow (e_{\sigma(1)}, ..., e_{\sigma(n)}) = (e_1, ..., e_n)P_{\sigma},
\ee
then the entries $p_{ij} = 1$ if $\sigma(j) = i$ (equivalently, $j = \sigma^{-1}(i)$) and $0$ otherwise. Thus,
\bea\label{e29} 
P_{\sigma}\;\mbox{defines an element of}\; \mbox{Aut}(\Gamma)\;\Leftrightarrow\;P_{\sigma}A = AP_{\sigma}.
\eea 
This can be seen as follows. The permutation $\sigma$ induces an automorphism of $\Gamma$ if and only if $a_{ij} = a_{\sigma(i)\sigma(j)}$, or equivalently, $a_{\sigma^{-1}(i)j} = a_{i\sigma(j)}$, $\forall\;i,j$. On the other hand,
\be
(P_{\sigma}A)_{ij} &=& \sum_k p_{ik}a_{kj} = p_{i\sigma^{-1}(i)}a_{\sigma^{-1}(i)j} = a_{\sigma^{-1}(i)j},\\
(AP_{\sigma})_{ij} &=& \sum_k a_{ik}p_{kj} = a_{i\sigma(j)}p_{\sigma(j)j}= a_{i\sigma(j)}.
\ee
So $a_{ij} = a_{\sigma(i)\sigma(j)},\forall\; i, j\Leftrightarrow P_{\sigma}A = AP_{\sigma}$.
\par
Since for any $\sigma\in S_n$, $e_{\sigma(i)}e_{\sigma(j)} = 0\; (i\ne j)$ is always true, so by (\ref{e27}), we have
\bea\label{e210}
\sigma\; \mbox{induces an automorphism of}\; \mathcal{E}\;\Leftrightarrow\;AP_{\sigma}^{(2)}=P_{\sigma}A. 
\eea
But $P_{\sigma}^{(2)} = (p_{ij}^2) =P_{\sigma}$ since $p_{ij} = 0$ or $1$. Thus (\ref{e29}) and (\ref{e210}) are equivalent, and hence every graph automorphism of $\Gamma$ induces an automorphism of $\mathcal E$. The second part of the theorem follows from part (2) of Theorem 2.2.
\end{proof}
\par
\section{Evolution algebras with given automorphism groups}
We now turn to the question of whether every finite group can be the automorphism group of an evolution algebra.
\begin{theorem} Let $\mathbb F$ be a field of characteristic $0$. Given any finite group $G$, there exists a finite dimensional idempotent evolution algebra $\mathcal E$ over $\mathbb F$ such that Aut$(\mathcal E)\cong G$. 
\end{theorem}
\begin{proof}
A well-known result due to Frucht \cite{Fru1} says that for any finite group $G$, there exists a graph $\Gamma$ such that Aut$(\Gamma)\cong G$. Suppose $\Gamma$ has $n$ vertices. Let the adjacency matrix of $\Gamma$ be $B$. For any nonnegative integer $x$, set $A(x) = B + xI_n$, where $I_n$ is the identity matrix of size $n$. Then for $\sigma \in S_n$, 
\bea\label{e31}
P_{\sigma}A(x) = A(x)P_{\sigma}\Leftrightarrow P_{\sigma}B = BP_{\sigma}.
\eea 
So by (\ref{e29}), the graph automorphism group of the graph corresponding to $A(x)$ is the same as that of $\Gamma$ for any nonnegative integer $x$. Now $\det(A(x))$ is a polynomial of degree $n$ in $x$, so it has at most $n$ roots in $\mathbb F$. Since $\mbox{char}(\mathbb F) = 0$, $\mathbb F$ contains a copy of $\mathbb Z$, thus there is a positive integer $m\in \mathbb{F}$ such that $A(m)$ is nonsingular. 
\par
Define an $n$-dimensional evolution algebra $\mathcal E$ by using $A(m)$ as the structure matrix together with a natural basis $e_1, ..., e_n$. Then $ \mathcal{E}$ is idempotent. Since all the diagonal entries of $A(m)$ are nonzero, the identity permutation $e$ satisfies $a_{e(1)1}\cdots a_{e(n)n}\ne 0$, so $t_{A(m)} = 1$ (see the paragraph just before Theorem 2.2), and thus Theorem 2.2 (1) implies that the subgroup $\mathcal{D}$ of diagonal automorphisms of $\mathcal E$ is trivial. Now (\ref{e210}), (\ref{e31}), and Theorem 2.3 together imply $\mbox{Aut}(\mathcal{E}) \cong \mbox{Aut}(\Gamma)\cong G$.
\end{proof}
\par
We denote by $\mathcal{E}(\Gamma)$ the evolution algebra defined by using the adjacency matrix of a graph $\Gamma$ as its structure matrix with respect to a natural basis. 
\par
\begin{example} 
Let $K_n$ be the complete graph with $n$-vertices without self-loop, then $\mbox{Aut}(K_n) \cong S_n$. Abusing notation, we also  denote the adjacency matrix of $K_n$ by $K_n$. Suppose that char$(\mathbb{F})$ does not divide $n-1$. For $n>1$, $K_n$ has $0$ on the diagonal and $1$ at all other places, thus $\det (K_n) = (-1)^{n-1}(n-1)\ne 0$. So for $n>1$, $\mathcal{E}(K_n)$ is idempotent. If $D = \mbox{diag}(d_1, ..., d_n)\in \mathcal{D}$, then for each pair $i\ne j$, since both the $(i,j)$ and the $(j,i)$ entries of $K_n$ are equal to $1$, we must have $d_i^2 = d_j$ and $d_i = d_j^2$ by (\ref{e28}). Thus all $d_i$'s are roots of $x^3 - 1$. 
\par
For $n=2$, $\mbox{Aut}(K_2) \cong S_2 = \mathbb{Z}_2$. If char$(\mathbb{F}) = 3$, then since $x^3 - 1 = (x-1)^3$, $\mathcal{D}$ is trivial, and $\mbox{Aut}(\mathcal{E}(K_2))\cong\mathbb{Z}_2$. If char$(\mathbb{F}) \ne 3$ and $x^3 - 1$ splits in $\mathbb F$, then $\mathcal{D}$ is generated by  $\mbox{diag}(\xi,\xi^2)$, where $\xi$ is a primitive root of  $x^3 - 1$, so $\mathcal{D}\cong \mathbb{Z}_3$. Let $\sigma$ be the generator of $\mbox{Aut}(K_2)$, then $\sigma^{-1}\mbox{diag}(\xi,\xi^2)\sigma =  \mbox{diag}(\xi^2,\xi)$, and so $\mbox{Aut}(\mathcal{E}(K_2)) \cong S_3$ in this case.
\par
For $n\ge 3$, by using distinct triples $i,j,k$ and (\ref{e28}), we see that $\mbox{diag}(d_1, ..., d_n)\in\mathcal{D}$ if and only if all $d_i = 1$. So $\mathcal{D}$ is trivial and $\mbox{Aut}(\mathcal{E}(K_n)) \cong S_n$ by Theorem 2.3.
\end{example}
Example 3.1 shows that for an $n$-dimensional idempotent evolution algebra $\mathcal E$, $\mbox{Aut}(\mathcal{E})$ can be bigger than $S_n$. We next give a classification of $n$-dimensional idempotent evolution algebras whose automorphism groups are exactly $S_n$. 
\begin{lemma} Let $\Gamma$ be a graph with $n$-vertices (self-loops are allowed) and let $A = (a_{ij})$ be its adjacency matrix. If $\mbox{Aut}(\Gamma)\cong S_n$, then $A = aK_n+bI_n$, where $a, b \in \{0,1\}$.
\end{lemma}
\begin{proof} This is clear. Under the assumption, either all vertices have no self-loop, or they all have; and if $a_{ij}\ne 0$ for some pair $i\ne j$, then for for any other pair $r\ne s$, there is a $\sigma\in S_n$ such that $\sigma(i) = r$ and $\sigma(j) = s$.
\end{proof}
\begin{lemma} For $n\ge 4$, if $\mathcal E$ is an $n$-dimensional idempotent evolution algebra such that $\mbox{Aut}(\mathcal{E})\cong S_n$, then the diagonal automorphism subgroup $\mathcal D$ is trivial and $\mbox{Aut}(\Gamma_{\mathcal{E}})\cong S_n$. 
\end{lemma}
\begin{proof}
For $n=4$, $S_4$ has two nontrivial normal subgroups: the alternating subgroup $A_4$ and $V_4 = \{(1), (12)(34), (13)(24), (14)(23)\}$, they both have even orders. Since the order of $\mathcal D$ is an odd number by Theorem 2.2 (1), $\mathcal D$ must be trivial. For $n\ge 5$, the only nontrivial normal subgroup of $S_n$ is the alternating subgroup $A_n$, so $\mathcal D$ is also trivial. Now Theorem 2.2 (2) completes the proof.
\end{proof}
\par
 Let $\mathcal E$ be an $n$-dimensional idempotent evolution algebra with a natural basis $e_1, ..., e_n$, and the structure matrix $A =(a_{ij})$. Assume that $\mathcal{G}=\mbox{Aut}(\mathcal{E}) = S_n$. Consider the cases $n < 4$. 
\par
The case $n=1$ is trivial. Let $n=2$. Since $\mathcal{D}\subset S_2 = \mathbb{Z}_2$ and $|\mathcal{D}|$ is an odd number, $\mathcal D$ must be trivial, and so $\mbox{Aut}(\Gamma_A)\cong \mathbb{Z}_2$. Thus $A =\begin{pmatrix} a & b\\ b & a\end{pmatrix},a\ne 0, a^2 - b^2 \ne 0$; or $a=0, b\ne 0$ and $x^3 - 1$ has only one root in $\mathbb{F}$. These evolution algebras are isomorphic to one of the evolution algebras given by the structure matrices  $C = \begin{pmatrix} 1 & c\\ c & 1\end{pmatrix}, c^2 \ne 1$, or isomorphic to the one given by the structure matrix $\begin{pmatrix} 0 & 1\\ 1& 0\end{pmatrix}$. In the former case, the isomorphism is given by $e_i \rightarrow ae_i', i=1,2$, where $(e_1', e_2')$ is the natural basis corresponding to the structure matrix $C$ and $c = b/a$. In the latter case, the isomorphism is given by $e_i \rightarrow be_i', i=1,2$. We will show in Theorem 3.2 that the isomorphism classes of these evolution algebras are represented by $\begin{pmatrix} 1 & c\\ c & 1\end{pmatrix}, c^2 \ne 1$ and  $\begin{pmatrix} 0 & 1\\ 1 & 0\end{pmatrix}$ (the structure matrices depend on the choice of the natural bases, so it needs to show that the algebras represented by these matrices are non-isomorphic). 
 \par
Now consider the case $n=3$. Let $\mbox{diag}(d_1,d_2,d_3)\in\mathcal{G}$. If at least two of the diagonal elements $a_{ii}, 1\le i\le 3$, of $A$ are nonzero, then by (\ref{e28}) and the fact that $A$ is nonsingular, we see that all $d_i = 1$, which implies that $\mathcal D$ is trivial and $\mbox{Aut}(\Gamma_{\mathcal{E}})\cong S_3$. That will imply (use (\ref{e29})) 
\bea\label{e32}
 A = \begin{pmatrix} a & b & b\\ b & a & b\\ b & b & a\end{pmatrix},\; a\ne 0, b\ne a, -a/2 \;(\mbox{char}(\mathbb{F})\ne 2).
\eea 
The isomorphism classes of these $3$-dimensional idempotent evolution algebras are represented by the following family of structure matrices (see Theorem 3.2):
\bea\label{e33}
 A = \begin{pmatrix} 1 & c & c\\ c & 1 & c\\ c & c & 1\end{pmatrix},\; c\ne 1, -1/2 \;(\mbox{char}(\mathbb{F})\ne 2).
\eea 
 The isomorphism is given by $e_i \rightarrow ae_i', 1\le i\le 3$, where $(e_1', e_2', e_3')$ is a natural basis of the evolution algebra with the structure matrix given by (\ref{e33}) and $c=b/a$.
\par
 Assume that there is only one $a_{ii}\ne 0$. Then without lost of generality, we can assume $a_{11} = a_{22} = 0$ and $a_{33}\ne 0$. If $\mathcal{D}$ is trivial, we again have $\mbox{Aut}(\Gamma_{\mathcal{E}})\cong S_3$, that will imply all $a_{ii}\ne 0$, a contradiction. So $\mathcal{D}$ is nontrivial, then as a normal subgroup of $S_3$, $\mathcal{D}\cong \mathbb{Z}_3$. Let $D=\mbox{diag}(d_1,d_2,d_3)$ be a generator of $\mathcal{D}$. Then $d_3 = 1$ since $a_{33}\ne 0$, and both $d_1$ and $d_2$ are primitive roots of $x^3-1$. Since $\mathcal{G}/\mathcal{D} \cong \mathbb{Z}_2$, there exists an element of the form $P_{(12)}D$ in $\mathcal{G}$ (see (\ref{e22})). By using $A(P_{(12)}D)^{(2)} = P_{(12)}DA$ (see (\ref{e210})), we see that $a_{12}\ne 0 \Leftrightarrow a_{21}\ne 0$; $a_{13}\ne 0 \Leftrightarrow a_{23}\ne 0$; and $a_{31}\ne 0 \Leftrightarrow a_{32}\ne 0$. Since we assumed that $a_{11} = a_{22} = 0$, we must have $a_{12}\ne 0$, otherwise $A$ would be singular. Also, since $d_3 = 1$, any of $a_{13}, a_{23}, a_{31}, a_{32}$ is nonzero would imply $d_1 = d_2 = 1$. So all these entries must be $0$. Now use $A(P_{(12)}D)^{(2)} = P_{(12)}DA$ again, we see under the assumption that $\mathcal{D}$ is nontrivial, we have 
 \bea\label{e34}
 A = \begin{pmatrix} 0 & a & 0\\ a & 0 & 0\\ 0 & 0 & b\end{pmatrix},\quad a\ne 0, b\ne 0.
 \eea
 All these evolution algebras are isomorphic to the one with $a=b=1$ by the map $e_i\rightarrow ae_i', i=1,2, e_3\rightarrow be_3'$.
 \par
 Assume that all $a_{ii} = 0$, let $D=\mbox{diag}(d_1,d_2,d_3)$ be a diagonal automorphism. Argue as before by using the fact that $a_{ij}\ne 0 $ implies $d_i = d_j^2$ and that $A$ is nonsingular, we see that $d_i = 1, 1\le i\le 3$. Thus $\mathcal D$ is trivial, and that leads to
 \bea\label{e35}
 A = \begin{pmatrix} 0 & a & a\\ a & 0 & a\\ a & a & 0\end{pmatrix},\quad a\ne 0.
 \eea
 So up to isomorphism, there is only one such evolution algebra represented by the one with $a=1$.
 \par
 We now ready to classify all $n$-dimensional idempotent evolution algebras $\mathcal E$ such that $\mbox{Aut}(\mathcal{E}) = S_n$.
 \begin{theorem} The following is a complete list of non-isomorphic $n$-dimensional idempotent evolution algebras whose automorphism group is $S_n$.  
 \begin{enumerate}
 \item For $n = 1$, there is only one isomorphic class given by the structure matrix $(1)$. 
 \item For $n = 2$, the non-isomorphic classes are represented by the structure matrices $\begin{pmatrix} 1 & c\\ c & 1\end{pmatrix}, c^2\ne 1$; and $\begin{pmatrix} 0 & 1\\ 1 & 0\end{pmatrix}$ (only if $x^3 - 1$ has one root in $\mathbb{F}$).
\item For $n = 3$, the non-isomorphic classes are represented by the following structure matrices (note that the class given by the matrix in the middle only exists if $x^3 - 1$ has $3$ distinct roots in $\mathbb F$):
\be
\qquad\begin{pmatrix} 1 & c & c\\ c & 1 & c\\ c & c & 1\end{pmatrix}, c\ne 1, -1/2\;(\mbox{char}(\mathbb{F})\ne 2),\\
\begin{pmatrix} 0 & 1 & 0\\ 1 & 0 & 0\\ 0 & 0 & 1\end{pmatrix},\; 
\begin{pmatrix} 0 & 1 & 1\\ 1 & 0 & 1\\ 1 & 1 & 0\end{pmatrix} (\mbox{char}(\mathbb{F})\ne 2).
\ee 
\item For $n \ge 4$, 
the non-isomorphic classes are represented by the following structure matrices ($\mbox{char}(\mathbb{F})\nmid n-1$ for the second and the third cases):
\be
 \qquad\begin{pmatrix} 1 & c & \cdots & c\\ c & 1 & \ddots & \vdots\\
 \vdots & \ddots & \ddots & c\\ c &\cdots & c & 1\end{pmatrix}, c\ne 1, 1/(1-n),\;
 \begin{pmatrix} 0 & 1 & \cdots & 1\\ 1 & 0 & \ddots & \vdots\\
 \vdots & \ddots & \ddots & 1\\ 1 &\cdots & 1 & 0\end{pmatrix}.
\ee
 \end{enumerate}
 \end{theorem}
 \begin{proof}  Let $\mathcal E$ be an $n$-dimensional idempotent evolution algebra with the structure matrix $A$ with respect to a natural basis $e_1, ..., e_n$ such that $\mbox{Aut}(\mathcal{E}) = S_n$. To prove the theorem, we first find the structure matrices for $n \ge 4$. By Lemma 3.1 and Lemma 3.2, we have
\bea\label{e36}
\qquad A = \begin{pmatrix} a & b & \cdots & b\\ b & a & \ddots & \vdots\\
\vdots & \ddots & \ddots & b\\ b &\cdots & b & a\end{pmatrix}\;\mbox{such that}\; \det(A) \ne 0.
\eea
Since $\det(A) = (a+(n-1)b)(a-b)^{n-1}$, we have $\det(A)\ne 0\Leftrightarrow a\ne b, (1-n)b$. Denote the matrix of (\ref{e36}) by $A(a,b)$ and the corresponding evolution algebra by $\mathcal{E}(a,b)$. If $a\ne 0$, then $\mathcal{E}(a,b)$ is isomorphic to $\mathcal{E}(1,b/a)$ by $e_i\rightarrow ae'_i, 1\le i\le n$. If $a = 0$, $\mathcal{E}(0,b)$ is isomorphic to $\mathcal{E}(0,1)$ by $e_i\rightarrow be'_i, 1\le i\le n$. 
\par
It remains to prove that $\mathcal{E}(1,c), c\ne 1, 1/(1-n)$, and $\mathcal{E}(0,1)$ are pairwise non-isomorphic evolution algebras.
\par
Let $K_n = (\gamma_{ij})$ (the adjacency matrix of the complete graph with $n$ vertices without self-loop). Then  $\gamma_{ii} = 0, 1\le i\le n$ and $\gamma_{ij} = 1, i\ne j$. Note that $A(1,c) = I_n + cK_n$ and $A(0,1) = K_n$.
\par
Assume that $\mathcal{E}(0,1)$ is isomorphic to some $\mathcal{E}(1,c)$. Then by Theorem 2.1 and equation (\ref{e21}), there exists a matrix $P=(p_{ij})$ and a permutation $\sigma\in S_n$ such that 
\be
p_{ij}\ne 0\Leftrightarrow i=\sigma(j)\;\mbox{and} \;K_nP^{(2)} = P+cPK_n. 
\ee
Thus for all $1\le i, j\le n$, from
\be
(K_nP^{(2)})_{ij} &=&\sum_{k=1}^n\gamma_{ik}p_{kj}^2 = 
\gamma_{i\sigma(j)}p^2_{\sigma(j)j}\quad\mbox{and}\\
(P+cPK_n)_{ij}&=&p_{ij} +c\sum_{k=1}^np_{ik}\gamma_{kj} = p_{ij}+c\gamma_{\sigma^{-1}(i)j},
\ee 
we have $\gamma_{i\sigma(j)}p^2_{\sigma(j)j}=p_{ij}+c\gamma_{\sigma^{-1}(i)j}$. Set $i = \sigma(j)$, we have $p_{\sigma(j)j} = 0$, which is a contradiction. Thus $\mathcal{E}(0,1)$ cannot be isomorphic to any $\mathcal{E}(1,c)$.
\par
Now assume that $\mathcal{E}(1,c)$ is isomorphic to $\mathcal{E}(1,b)$ for some $c\ne b$. Again apply Theorem 2.1 and equation (\ref{e21}), let $P=(p_{ij})$ and $\sigma\in S_n$ be such that 
\be
p_{ij}\ne 0\Leftrightarrow i=\sigma(j)\; \mbox{and}\; P^{(2)}+bK_nP^{(2)} = P+cPK_n.
\ee 
Then similar to the discussions above, for all $1\le i, j\le n$, we have $p_{ij}^2+\gamma_{i\sigma(j)}p^2_{\sigma(j)j}=p_{ij}+c\gamma_{\sigma^{-1}(i)j}$. 
Choose $i = \sigma(j)$, we have $p^2_{\sigma(j)j} = p_{\sigma(j)j}$, which implies $p_{\sigma(j)j} = 1, 1\le j\le n$, that is, $P$ is the permutation matrix of $\sigma$. Thus since $P^{(2)} = P$, we have $bK_nP = cPK_n$. But that would lead to a contradiction since $c\ne b$ and $K_nP = PK_n$ by (\ref{e29}). Therefore $\mathcal{E}(1,c), c\ne 1, 1/(1-n)$ are pairwise non-isomorphic. Finally, note that the proof covers the case $n = 2, 3$ too (recall that we only need to show the listed matrices define non-isomorphic algebras for these two cases).
\end{proof}
\section{Diagonal automorphism subgroup of an idempotent evolution algebra}
\par
Let $\mathcal{E}$ be an idempotent evolution algebra with a natural basis $e_1, ..., e_n$ and the structure matrix $A=(a_{ij})$, and let $\mathcal{D}$ be the diagonal automorphism subgroup of $\mbox{Aut}(\mathcal{E})$. We consider the problem of which idempotent evolution algebra $\mathcal{E}$ processes the maximal possible $\mathcal{D}$. To simplify our discussions, we assume that $\mathbb{F}$ is algebraically closed of characteristic $0$ in this section. 
\par
Let $\sigma\in S_n$ be such that $a_{\sigma(j)j}\ne 0, 1\le j\le n$. Suppose that 
\be
\sigma =\sigma_1\sigma_2\cdots\sigma_p= (i_1\cdots i_r)(j_1\cdots j_s)\cdots(p_1\cdots p_t)
\ee 
is the decomposition of $\sigma$ into disjoint cycles. For $D =\mbox{diag}(d_1, ..., d_n)$ in $\mathcal{D}$, we can divide $\{d_1, ..., d_n\}$ into disjoint subsets 
\be
\{d_{i_1}, ..., d_{i_r}\}, \{d_{j_1}, ..., d_{j_s}\}, ..., \{d_{p_1}, ..., d_{p_t}\}
\ee 
according to the decomposition of $\sigma$. Then by (\ref{e28}), we see that the $d_{i_h}$'s are roots of $x^{2^r - 1} - 1$, the $d_{j_h}$'s are roots of $x^{2^s - 1} - 1$, and so on. Furthermore, we have $d_{i_{h}}^2 = d_{\sigma_1({i_h})} = d_{i_{h+1}}$, and so on.
\par
If $\tau\in S_n$ is another permutation such that $a_{\tau(j)j}\ne 0, 1\le j\le n$, let $\tau =\tau_1\tau_2\cdots\tau_q$ be its disjoint decomposition. If, say, for some $1\le h\le r$, $i_h$ appears in $\tau_1$, then $d_{i_h}$ is also a root of $x^{2^t - 1} - 1$, where $t = o(\tau_1)$. This will imply that $d_{i_h}$ is a common root of $x^{2^r - 1} - 1$ and $x^{2^t - 1} - 1$, and so it is a root of $x^{d} - 1$, where $d = \gcd(2^r-1,2^t-1)$. Therefore, for a fixed $n$, the maximal possible $D$ occur when there exists only one $\sigma$ such that $a_{\sigma(j)j}\ne 0, 1\le j\le n$. This can also be seen by the fact that more nonzero entries of $A$ would impose more constraints on the possible value of the $d_i$'s. Thus, given $n$, the maximal possible $\mathcal{D}$ occur among those $\mathcal{E}$ such that the structure matrices satisfy $a_{ij}\ne 0 \Leftrightarrow i = \sigma(j)$ for some $\sigma\in S_n$. 
\par
Let $A$ be a structure matrix such that $a_{ij}\ne 0 \Leftrightarrow i = \sigma(j)$ for some fixed $\sigma\in S_n$. If $\sigma=\sigma_1\sigma_2\cdots\sigma_p$ ($p\ge 1$) is the decomposition of $\sigma$ into the product of disjoint cycles cycles of length $\ge 2$,  let $o(\sigma_i) = n_i, 1\le i\le p$. Then $\sum_{i=1}^pn_i \le n$. From the above discussions, we see that in this case, the maximal possible order of $\mathcal{D}$ is 
\be
(2^{n_1}-1)(2^{n_2} - 1)\cdots (2^{n_p} - 1) &< &2^{n_1}2^{n_2} \cdots (2^{n_p} -1)\\
{} &<& 2^{\sum_{i=1}^pn_i} - 1 \le 2^n - 1. \nonumber
\ee
On the other hand, if $\sigma$ is a cycle of length $n$, then from $d_j^2 = d_{\sigma(j)}$, $1\le j\le n$, we see that $D$ is determined by $d_1$. Since we assume that $\mathbb{F}$ is algebraically closed of characteristic $0$, in this case, $\mathcal{D}$ is isomorphic to the cyclic group $C_{2^n-1}$ formed by the roots of $x^{2^n-1}-1$. 
\begin{theorem}
Assume the base field $\mathbb F$ is algebraically closed of characteristic $0$. Let $\mathcal{E}$ be an idempotent evolution algebra with a natural basis $e_1, ..., e_n$ and the structure matrix $A=(a_{ij})$, and let $\mathcal{D}$ be the diagonal automorphism subgroup of $\mbox{Aut}(\mathcal{E})$. 
\begin{enumerate}
\item The maximal possible order of $\mathcal{D}$ is $2^n - 1$. This maximal order is achieved if and only $a_{ij}\ne 0 \Leftrightarrow i = \sigma(j)$ for some fixed cyclic permutation $\sigma\in S_n$ of length $n$; and in this case, $\mathcal{D}$ is cyclic of order $2^n - 1$ and $\mbox{Aut}(\mathcal{E})\cong C_{2^n-1}\rtimes C_n$.
\item All $n$-dimensional idempotent $\mathcal{E}$ with maximal $\mathcal{D}$ are isomorphic to the one represented by the structure matrix $A = P_{\sigma}$, where $\sigma=(12\cdots n)$.
\end{enumerate}
\end{theorem}
\begin{proof}
It remains to prove (2) and $\mbox{Aut}(\mathcal{E}(P_{\sigma}))\cong C_{2^n-1}\rtimes C_n$, where $\sigma=(12\cdots n)$. Since $\Gamma_{P_{\sigma}}$ is a cyclic digraph, $\mbox{Aut}(\Gamma_{P_{\sigma}}) \cong C_n$. So Theorem 2.2 implies $\mbox{Aut}(\mathcal{E}(P_{\sigma}))\cong C_{2^n-1}\rtimes C_n$. 
\par
To prove (2), let $B$ be the structure matrix of an $n$-dimensional evolution algebra with a diagonal automorphism subgroup of order $2^n-1$. Then there exists $b = (b_1, ..., b_n)^T\in \mathbb{F}^n$ such that $b_1\cdots b_n \ne 0$ and $B= P_{\tau}\mbox{diag}(b_1, ..., b_n)$ for some cyclic permutation $\tau$ of length $n$. By relabeling if necessary, we can assume that $\tau = \sigma = (12\cdots n)$. Thus we need to prove that $\mathcal{E}(B) \cong \mathcal{E}(P_{\sigma})$ for $B= P_{\sigma}\mbox{diag}(b_1, ..., b_n)$. By Theorem 2.1, it suffices to find a nonsingular diagonal matrix $D =\mbox{diag}(d_1, ..., d_n)$ such that $BD^{(2)} = DP_{\sigma}$. This is equivalent to solving the following equation for the $d_i$'s:
\be
P_{\sigma}\mbox{diag}(b_1, ..., b_n)\mbox{diag}(d_1^2, ..., d^2_n) = \mbox{diag}(d_1, ..., d_n)P_{\sigma}.
\ee
That is, to solve
\be
\mbox{diag}(b_1d_1^2, ..., b_nd_n^2) &=& P_{\sigma}^{-1}\mbox{diag}(d_1, ..., d_n)P_{\sigma} \\
                                                         {}  &=& \mbox{diag}(d_{\sigma(1)}, ..., d_{\sigma(n)})\\
                                                         {}  &=& \mbox{diag}(d_2, ..., d_{n}, d_1).
\ee
Note that if we know $d_1$, then we can find all $d_i, i > 1$, by
\bea\label{e42}
\qquad d_2 = b_1d_1^2, \; d_3 = b_2d_2^2 = b_2b_1^2d_1^4, \;...,\;  d_n = (\prod_{i=1}^{n-1}b_i^{2^{n-i-1}})d_1^{2^{n-1}}.
\eea
The constraint for $d_1$ is 
\be
d_1 = b_{n}b_{n-1}^2\cdots b_1^{2^{n-1}}d_1^{2^{n}} = (\prod_i^nb_i^{2^{n-i}})d_1^{2^n}.
\ee
Thus, if we take a root of the polynomial $(\prod_i^nb_i^{2^{n-i}})x^{2^n-1} - 1$ as $d_1$, then we can find a nonsingular diagonal matrix $D$ that satisfies $BD^{(2)} = DP_{\sigma}$ by (\ref{e42}). 
\end{proof}
\par\medskip
\section*{Acknowledgements}
Sriwongsa acknowledges the financial supports of the Center of Excellence in Theoretical and Computational Science (TaCS-CoE), Faculty of Science, KMUTT and the Thailand Science Research and Innovation (TSRI)  Basic Research  Fund: Fiscal year 2022 (FF65).  
\par
Zou acknowledges the support of a Simons Foundation Collaboration Grant for Mathematicians (416937), and thanks Nanning Normal University for its support (through the grants NNSFC (11961050) and GNSF (2020GXNSFAA159053)) during his visit there.
 
\end{document}